\documentclass[12pt]{amsart}
\usepackage{a4}
\usepackage{color}
\usepackage{graphicx}
\usepackage{subfigure}
\usepackage[colorlinks=true,
linkcolor=webgreen,
filecolor=webgreen,
citecolor=webgreen]{hyperref}
\definecolor{webgreen}{rgb}{0,0,1}
\definecolor{recrown}{rgb}{1,.2,.6}

\usepackage{fullpage}
\begin{document}
\newtheorem{theorem}{Theorem}
\newtheorem{corollary}[theorem]{Corollary}
\newtheorem{lemma}[theorem]{Lemma}
\theoremstyle{definition}
\newtheorem*{example}{Example}
\newtheorem*{examples}{Examples}
\newtheorem*{notation}{Notation}
\theoremstyle{remark}
\newtheorem*{remarks}{\bf Remarks}
\theoremstyle{thmx}
\newtheorem{thmx}{\bf Theorem}
\renewcommand{\thethmx}{\text{\Alph{thmx}}}
\newtheorem{lemmax}{Lemma}
\renewcommand{\thelemmax}{\text{\Alph{lemmax}}}
\theoremstyle{thmx}
\newtheorem{property}{\bf Property}
\renewcommand{\theproperty}{\text{\Alph{property}}}
\theoremstyle{definition}
\newtheorem*{definition}{Definition}
\newtheorem*{remark}{\bf Remark}
\title{\bf Metrically Round and Sleek Metric Spaces}
\markright{}

\author{Jitender Singh$^{\dagger,*}$}
\address{$~^\dagger$ Department of Mathematics, Guru Nanak Dev University, Amritsar-143005, India}
\author{T. D. Narang$^\ddagger$}
\address{$~^\ddagger$ Department of Mathematics, Guru Nanak Dev University, Amritsar-143005, India}
\footnotetext[3]{$^{,*}$Corresponding author: jitender.math@gndu.ac.in}
\footnotetext[2]{tdnarang1948@yahoo.co.in}
\date{}
\maketitle
\parindent=0cm
\begin{abstract}
A  round metric space is the one in which closure of each open ball is the corresponding closed ball. By a  sleek metric space, we mean a metric space in which  interior of each closed ball is the corresponding open ball. In this, article we establish some  results on  round metric spaces and sleek metric spaces.
\end{abstract}
\section{Introduction.}
All normed linear spaces are known to have the following two properties with the metric induced by the norm:
\begin{property}\label{P1}
Closure of every open ball is the corresponding closed ball.
\end{property}
\begin{property}\label{P2}
Interior of every closed ball is the corresponding open ball.
\end{property}
Properties \ref{P1} and \ref{P2} may not hold in metric spaces or even in linear  metric spaces (see for example, \cite{artemiadis,wong}). However, starting with the work of Art\'emiadis \cite{artemiadis}, there are some necessary and sufficient conditions known in the the literature for Property \ref{P1} or \ref{P2} to hold  in metric spaces (see \cite{Na,TH,JSTD2020}).  For a metrizable space, Nathanson \cite[p.~738]{Na} defined a round metric as the one for which Property \ref{P1} holds.  A metrizable space whose topology is induced by a round metric was then defined to be a \emph{ round metric space}.  In fact, Nathanson \cite{Na} obtained several interesting classes of  round metric spaces (see Theorems \ref{Na1}-\ref{Na4} below). On the other hand,  Kiventidis  \cite{TH} was the first to discuss metric spaces having Property \ref{P2}. Recently, Singh and Narang \cite{JSTD2020}  have discussed metric spaces and linear metric spaces which have  Property \ref{P2} under some convexity conditions.
Analogously, for a metrizable space, we define \emph{sleek metric} as the one for which Property \ref{P2} holds. A metrizable space whose topology is induced by a sleek metric will be called a \emph{sleek metric space}. In the present paper, in addition to exploring metric spaces, linear metric spaces, and subspaces for Properties \ref{P1}  or \ref{P2}, we also investigate the behavior of these properties under taking union, intersection, and product. First we fix some notations.
Throughout, the symbol $X$ will denote a metrizable topological space with at least two points.
Let $d$ be a metric inducing the topology of $X$. For any $x\in X$ and $r>0$, the sets $B_d(x,r)=\{y\in X~|~d(x,y)<r\}$ and  $B_d[x,r]=\{y\in X~|~d(x,y)\leq r\}$, respectively denote the open and the closed balls in $X$.
We also denote by $\bar{B}_d(x,r)$,  the closure of $B_d(x,r)$ in $X$ and $B^\circ_d[x,r]$, the interior of $B_d[x,r]$ in $X$. For any two subsets $Y$ and $Z$ of $X$, we use $Y\setminus Z$ to denote the set $\{y\in Y~|~y\not\in Z\}$.  Also, any point of the cartesian product $Y\times Z$ will be denoted by $y\times z$, where $y\in Y$ and $z\in Z$. Finally,
for each natural number $n$, the symbol $\rho_n$ will denote the Euclidean metric of $\mathbb{R}^n$.
\begin{examples}
1.  The metric $\rho_2$ of $\mathbb{R}^2$ is round for the subspaces shown in Fig.~\ref{F1a}-\ref{F1c}, but $\rho_2$ is not
round for the subspace shown in Fig.~\ref{F1d}. It may be remarked that the shapes corresponding to Fig.~\ref{F1abc}(a)  and (b) are round in accordance with the geometric intuition but Fig.~\ref{F1abc}(c) is not round as per the geometric perception.
\begin{figure}[h!!!]
\subfigure[{}]{\includegraphics[width=0.25\textwidth]{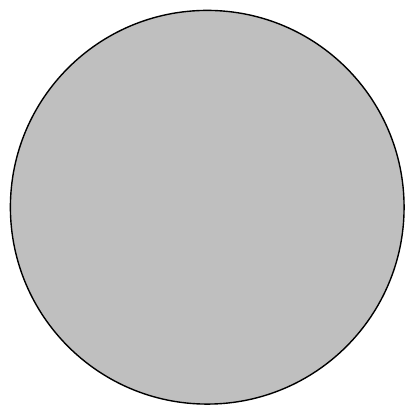}\label{F1a}}%
\subfigure[{}]{\includegraphics[width=0.25\textwidth]{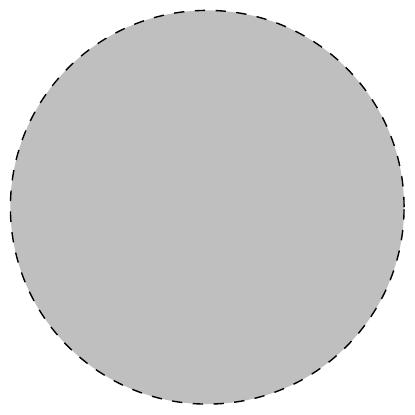}\label{F1b}}%
\subfigure[{}]{\includegraphics[width=0.25\textwidth]{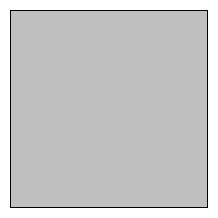}\label{F1c}}%
\subfigure[{}]{\includegraphics[width=0.25\textwidth]{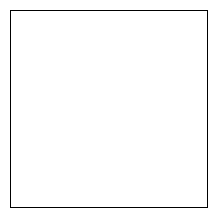}\label{F1d}}
\caption{} \label{F1abc}
\end{figure}

2. The metric $\rho_1$ is round for $\mathbb{R}$ as well as  for each of the subspaces  $\mathbb{Q}$, $(0,1)$, and $[0,1]$.

3.  No metric equivalent to $\rho_1$ is round for the subspace $[0,1]\cup [2,3]$ of $\mathbb{R}$. (see Theorem \ref{Na1} below).

4.  The metric $\rho_2$ of $\mathbb{R}^2$ is round for each of the subspaces $S^1=\{x\times y\in \mathbb{R}^2~|~x^2+y^2=1\}$ and $[-1,1]\times {0}$ but $\rho_2$ is not round for the subspace $S^1\cap ([-1,1]\times 0)$, since this intersection is the two point set as shown in Fig.~\ref{F2}.
\begin{figure}[h!!!]
\centering
 \includegraphics[width=0.50\textwidth]{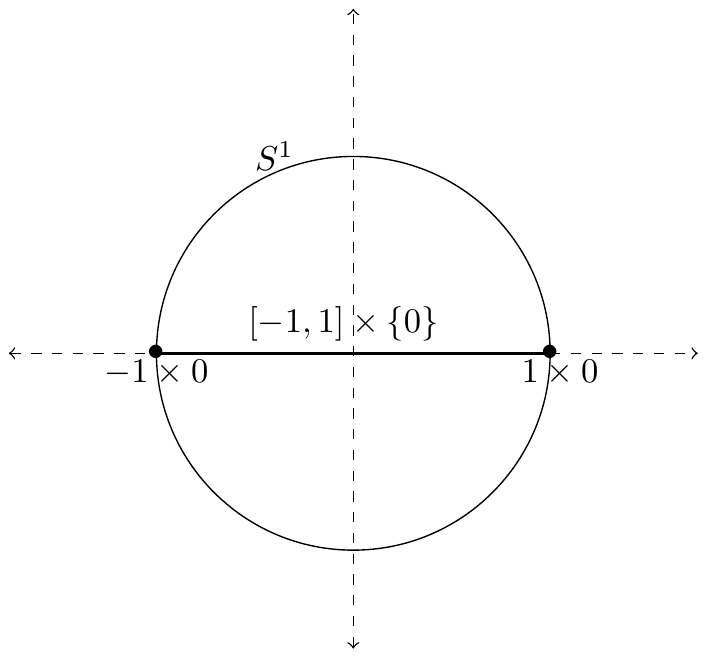}
\caption{} \label{F2}
\end{figure}
\end{examples}
The last two  examples show that the property of subspaces being  round in a given metric space is not always preserved under union and intersection. It was remarked in \cite{TH} that an open or dense subspace of a  round metrizable space is  round.
Nathanson \cite[Theorems 1-4]{Na} established the following main results for the  round metric spaces:
\begin{thmx}\label{Na1}
Let $X = A \cup K$ be a metrizable space, where $A$ and $K$ are nonempty, disjoint, closed sets, and $K$ is compact. Then no metric for $X$ is round.
\end{thmx}
\begin{thmx}\label{Na2}
Let $(X,d_1)$ and $(Y,d_2)$ be metric spaces without isolated points, and let $f:X\rightarrow Y$ be a surjection such that for $x,y,z,\in X$, if $d_1(x,z)<d_1(x,y)$, then $d_2(f(x),f(z))<d_2(f(x),f(y))$. If $d_1$ is a round metric for $X$, then $d_2$ is a round metric for $Y$.
\end{thmx}
\begin{thmx}\label{Na3}
Let $(X,d)$ be a metric space. Then there exists an equivalent metric on $X$ that is bounded but not round.
\end{thmx}
\begin{thmx}\label{Na4}
Let $\{(X_k,d_k)\}_{k=1}^\infty$ be a countable family of  metric spaces such that  $\text{diam}(X_k)<\infty$ for all but finitely many $k$. The product space $X=\prod_{k}X_k$ is metrizable. If $x=(x_k)$, $y=(y_k)\in X$, define
\begin{eqnarray}\label{ee1}
D(x,y) &=& \sum_{k=1}^\infty {d_k(x_k,y_k)}/{(\lambda_k2^k)},
\end{eqnarray}
where $\lambda_k=1$ if $\text{diam}(X_k)=\infty$ and $\lambda_k=\text{diam}(X_k)$ if $\text{diam}(X_k)<\infty$.  Then metric $D$ is a metric for $X$.  The metric  $D$ is round for $X$ if and only if the metric $d_k$ is round for $X_k$ for all $k$.
\end{thmx}
Interestingly, as we shall see in this paper, the analogues of Theorems \ref{Na1}-\ref{Na4} hold for  sleek metric spaces too.

Now coming to Property \ref{P2}, one observes that among the various subspaces of $\mathbb{R}^2$ as shown in Fig.~\ref{F1abc},  the metric $\rho_2$ is sleek only for the subspace \ref{F1b}. None of the two subspaces given in  Fig.~\ref{F1c} or \ref{F1d} is sleek with respect to the metric $\rho_2$ due to the presence of corner points. This observation is also in accordance with the intuitive notion of \emph{geometric} sleekness. On the other hand, the non-sleek case \ref{F1a} is  not in agreement with the geometric intuition.
The next two examples show that as in the  round  metric spaces, metric sleekness too is not always preserved under union and intersection.
\begin{examples}
1.  We consider the subspaces
$X_1=\{\cos t\times \sin t~|~t\in (-\pi/4,\pi/4)\}$ and $X_2=\{\cos t\times \sin t~|~t\in (3\pi/4,5\pi/4)\}$
of $\mathbb{R}^2$ with the metric $\rho_2$ as shown in Fig.~\ref{F3a}.
We observe that the metric $\rho_2$ is sleek for $X_1$ as well as $X_2$. Now consider the metric $\rho_2$ for the subspace $Z=X_1\cup X_2$. Here, \begin{eqnarray*}
B_{\rho_2}[1\times 0,2]\cap Z=Z\neq (B_{\rho_2}[1\times 0,2]\cap Z).
\end{eqnarray*}
Consequently, interior of the closed ball $B_{\rho_2}[1\times 0,2]\cap Z$ in $Z$ is $Z$ which is not equal to the corresponding open ball $B_{\rho_2}(1\times 0,2)\cap Z$. Thus, the metric $\rho_2$ is not sleek for $Z$.

2. Let $Y_1=\mathbb{R}\times (-\infty,0]$ and $Y_2=[0,\infty)\times \mathbb{R}$.
The metric $\rho_2$ is sleek for each of the subspaces $Y_1$ and $Y_2$ of $\mathbb{R}^2$.  However, the metric $\rho_2$ is not sleek for the subspace $Y_1\cap Y_2=[0,\infty)\times (-\infty,0]$ as can be visualized from Fig.~\ref{F3b} as the darker-shaded region.
We see that $0\times 0 \in (B_{\rho_2}^\circ[1\times -1,\sqrt{2}]\cap Y_1\cap Y_2)$, but $0\times 0\not\in (B_{\rho_2}(1\times -1,\sqrt{2})\cap Y_1\cap Y_2)$. 
\begin{figure}[h!!!]
\centering
\subfigure[]{\includegraphics[width=0.4\textwidth]{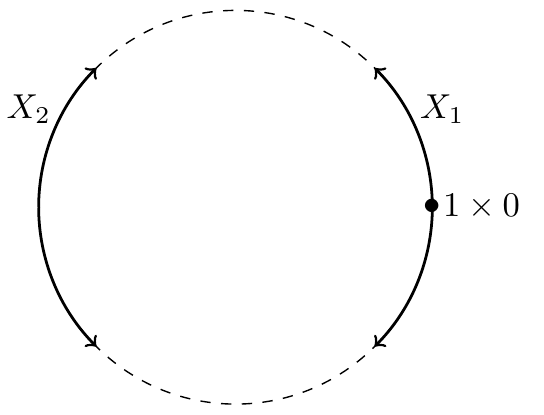}\label{F3a}}%
\subfigure[]{\includegraphics[width=0.55\textwidth]{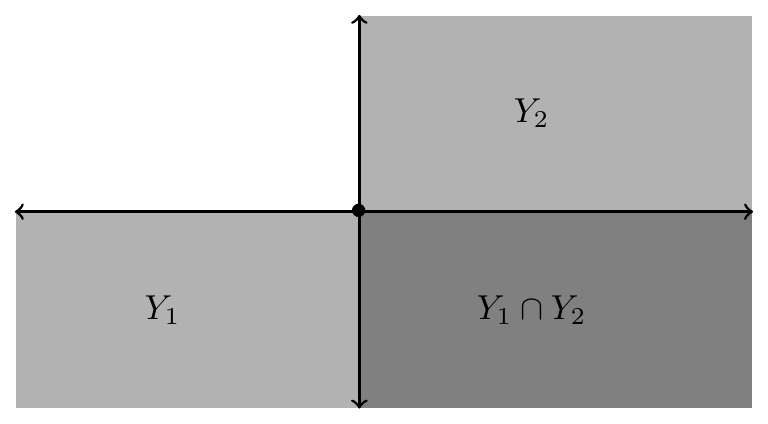}\label{F3b}}
\caption{}
\end{figure}
\end{examples}
Among the subspaces $(0,1)$, $[0,1]$, and $[0,1]\cup [2,3]$ of real line, the metric $\rho_1$ is round and sleek for $(0,1)$, round but not sleek for $[0,1]$, and neither round nor sleek for $[0,1]\cup [2,3]$. Now
consider the subspace $X'=\mathbb{R}\times \{0,1\}$ of $\mathbb{R}^2$ with the metric $\rho_2$. For any point $a\times b$ in $X'$ and $r>0$, we find that
\begin{eqnarray*}
B_{\rho_2}[a\times b,r]=\begin{cases} C:=\{x\times b~\big{|}~|x-a|\leq r\}, &~\text{if}~0<r<1;\\
C\cup \{a\times (1-b)\},&~\text{if}~r=1;\\
C\cup \{x\times (1-b)~\big{|}~|x-a|\leq \sqrt{r^2-1}\},&~\text{if}~1<r.
\end{cases}
\end{eqnarray*}
Each of these three type of closed balls are shown in Fig.~\ref{F4} as the thick parts. In the second case (b), we have
\begin{figure}[h!!!]
 \centering
 \subfigure[$0<r<1$]{\centering\includegraphics[width=0.48\textwidth]{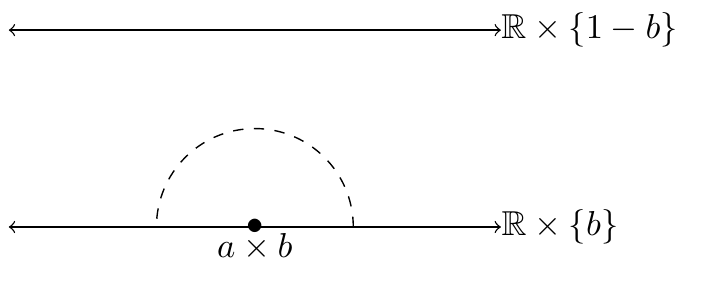}}%
 \subfigure[$r=1$]{\centering\includegraphics[width=0.48\textwidth]{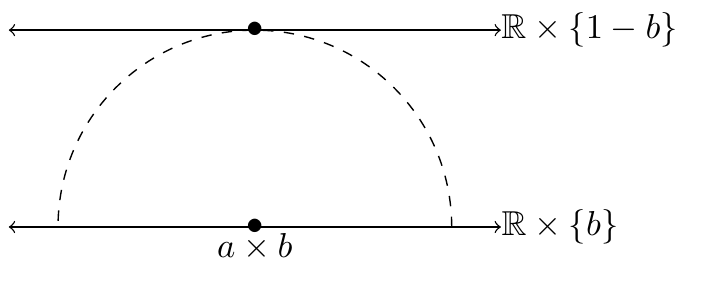}}
 \subfigure[$r>1$]{\centering\includegraphics[width=0.53\textwidth]{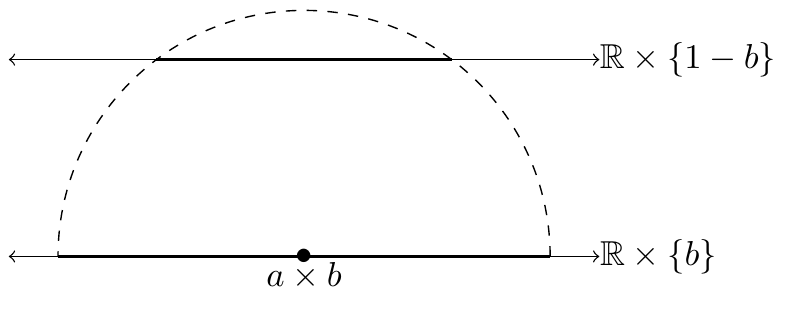}}
    \caption{}\label{F4}
\end{figure}
$B_{\rho_2}(a\times (1-b),1/2)\cap{B}_{\rho_2}(a\times b,1)=\emptyset$, and therefore,  $a\times (1-b)\not\in \bar{B}_{\rho_2}(a\times b,1)$. Thus, the metric ${\rho_2}$ is not round for $X'$.
Also, if either $0<r<1$, or $1<r$, then clearly we have $B_{\rho_2}^\circ[a\times b,r]=B_{\rho_2}(a\times b,r)$. Now consider the case $r=1$. The point $a\times (1-b)$ is not an interior point of $B_{\rho_2}[a\times b,1]$, since for every $\epsilon>0$, $B_{\rho_2}(a\times (1-b),\epsilon)$ always contains the point $z=(a+\epsilon/2)\times (1-b)$, where ${\rho_2}(a\times b,z)=\sqrt{1+\epsilon^2/4}>1$. Thus, $B_{\rho_2}^\circ[a\times b,1]=B_{\rho_2}(a\times b,1)$. We conclude that the metric $\rho_2$ is sleek for $X'$.

Thus, in general, the notions of  round metric space and  sleek metric space are different.

Metric spaces satisfying certain convexity conditions turn out to be  round, or sleek, or round as well as sleek. We mention below some of such known results.
\begin{remarks} Let $(X,d)$ be a metric space.

1. If $(X,d)$ is \emph{$\lambda$-convex}, that is,  it has the property that for any  $x,y\in X$ and a fixed $\lambda\in (0,1)$, there exists $z\in X$ such that $d(z,x)=(1-\lambda)d(x,y)$ and $d(z,y)=\lambda d(x,y)$, then the metric $d$ is round for $X$ (see \cite[Theorem 5]{Na}).

2. Let $(X,d)$ be complete; and let $(X,d)$ be \emph{metrically convex} or \emph{convex}, that is, it has the property that for every pair of distinct points $x$ and $y$ in $X$, there exists $z\in X$ such that $z\not\in \{x,y\}$ and $d(x,z)+d(z,y)=d(x,y)$.  If $(X,d)$ is \emph{externally convex}, that is, it has the property that  for every pair of distinct points $x$ and $y$ of $X$, there exists $z\in X$ different from $x$ and $y$  such that $d(x,y)+d(y,z)=d(x,z)$, then the metric $d$ is sleek for $X$ (see \cite[Theorem 2.4]{JSTD2020}.

3. If $(X,d)$ is \emph{strongly externally convex}, that is, it has the property that for every pair of distinct points $x$ and $y$ in $X$ and any $s>d(x,y)$, there exists a point $z$ in $X$ such that $d(x,y)+d(y,z)=d(x,z)=s$, then the metric $d$ is sleek for $X$ (see \cite[Theorem 2.5]{JSTD2020}).
\end{remarks}
\section{Some properties of round and sleek metric spaces.}
We begin by proving the following characterization of non-round metrics.
\begin{theorem}\label{t1}
Let $(X,d)$ be a metric space. The metric $d$ is not round for $X$ if and only if there exists an open set $U$ and  $x\in X\setminus U$ such that the map $d(x,\cdot):U\rightarrow \mathbb{R}$ has a minimum.
\end{theorem}
\begin{proof}
Suppose the metric $d$ is not round for $X$. Then there exists a pair of distinct points $u$ and $v$ in $X$ for which $v\not\in \bar{B}_d(u,d(u,v))$. Consequently, there exists an $\epsilon>0$, such that $B_d(v,\epsilon)\cap {B}_d(u,d(u,v))=\emptyset$, and so, $B_d(v,\epsilon)\subset X\setminus\bar{B}_d(u,d(u,v))$. Then $U=X\setminus\bar{B}_d(u,d(u,v))$ is an open set in $X$ such that $u\not\in U$,  $v\in U$, and the map $d(u,\cdot):U\rightarrow \mathbb{R}$ has the minimum  $d(u,v)$.

Conversely, let there be an open set $U$ in $X$ and a point $x\in X\setminus U$, such that the map $d(x,\cdot):U\rightarrow \mathbb{R}$ has a minimum value. Suppose $y\in U$ for which $d(x,y)=\inf\{d(x,z)~|~z\in U\}$. Since for every $t\in B_d(x,d(x,y))$, we have $d(x,t)<d(x,y)$, this in view of the fact that $d(x,y)$ is the minimum value of $d(x,\cdot)$ on $U$ gives $t\not\in U$. Consequently, $B_d(x,d(x,y))\cap U=\emptyset$, and so, $B_d(x,d(x,y))\subset X\setminus U$. Then $\bar{B}_d(x,d(x,y))\subseteq \overline{X\setminus U}=X\setminus U$, and since $y\not\in X\setminus U$, we find that $y\not\in \bar{B}_d(x,d(x,y))$. So, $d$ is not a round  metric for $X$.
\end{proof}
\begin{corollary}
A metrizable space with at least two points and having an isolated point is not metrically round.
\end{corollary}
\begin{proof} Let $X$ be such a space, and let
the metric $d$ induces the topology of $X$. If $y\in X$ is an isolated point, then the  set $\{y\}$ is open in $X$. So, for any $x\in X\setminus \{y\}$, the map $d(x,\cdot): \{y\}\rightarrow \mathbb{R}$ has the minimum value $d(x,y)$.
\end{proof}
A subset $S$ of a metric space $(X,d)$ is said to be a \emph{metric segment} joining two distinct points $x$ and $y$ of $X$ if there exists a closed interval $[a,b]$ in real line and an isometry $\gamma:[a,b]\rightarrow (X,d)$ which maps $[a,b]$ onto $S$ with $\gamma(a)=x$ and $\gamma(b)=y$. If the metric space $(X,d)$ has the property that every pair of points can be joined by a metric segment, then $(X,d)$ is convex (see \cite[p.~35]{b1}). The converse known as, the fundamental theorem of metric convexity, is often true, and it states that in a complete convex metric space, any two distinct points can be joined by a metric segment. This result was proved in \cite{menger} by Karl Menger, one of the pioneers in the study of metric spaces.
\begin{theorem}\label{t2}
If a metric space $(X,d)$ has the property that every pair of distinct points can be joined by a metric segment, then the metric $d$ is round for $X$.
\end{theorem}
\begin{proof}
If possible, suppose that the metric $d$ is not round for $X$. By Theorem \ref{t1}, there exists an open set $U$ and $x\in X\setminus U$ such that the map $d(x,\cdot):U\rightarrow \mathbb{R}$ has the minimum $d(x,y)$ for $y\in U$. Then $x\neq y$, and by the hypothesis, there exists a closed interval $[a,b]$ in real line and an isometry $\gamma:[a,b]\rightarrow X$, such that $\gamma(a)=x$, $\gamma(b)=y$. Since $\gamma$ is an isometry, we have  $d(x,y)=d(\gamma(a),\gamma(b))=b-a>0$. By continuity of $\gamma$, the inverse image $\gamma^{-1}(U)$ is an open set containing $b$, and for all $t\in \gamma^{-1}(U)\subseteq [a,b]$, we have
\begin{eqnarray*}
 b-a=d(x,y)\leq d(x,\gamma(t))=d(\gamma(a),\gamma(t))=t-a\leq b-a,
\end{eqnarray*}
and so, $t=b$. Consequently, $\gamma^{-1}(U)=\{b\}$ which is not open in $[a,b]$. This contradicts the continuity of $\gamma$.
\end{proof}
The fundamental theorem of metric convexity together with Theorem \ref{t2} implies that every complete convex metric space is metrically round. This result has been proved earlier in \cite{wong,TH}, however, our approach is different.

In \cite{JSTD2020}, Singh and Narang obtained the following characterization of  sleek metric spaces, which will be used in the sequel.
\begin{thmx}\label{thA}
A metric space  $(X,d)$ is metrically sleek if and only if for any $x\in X$ and $r>0$, each $y\in B_d[x,r]$ satisfying $d(x,y)=r$ is a limit point of the set $X\setminus B_d[x,r]$.
\end{thmx}
The following result presents an analogue of Theorem \ref{t1} for non-sleek metrics.
\begin{theorem}\label{th4}
Let $(X,d)$ be a metric space. The metric $d$ is not sleek for $X$ if and only if there exists $x\in X$ and an open set $U$ containing $x$, such that the map $d(x,\cdot):U\rightarrow \mathbb{R}$ has a maximum.
\end{theorem}
\begin{proof}Suppose there exists $x\in X$ and an open set $U$ containing $x$ such that
$d(x,y)=\max\{d(x,z)~|~z\in U\}$ for some $y\in U$.
Then $U\subseteq B_d[x,d(x,y)]$, and since $y\in U$, we have $y\in B_d^\circ[x,d(x,y)]$. Consequently, $y$ is not a limit point of the set $X\setminus B_d[x,d(x,y)]$. So, by Theorem \ref{thA}, $d$ is not a sleek metric for $X$.

Conversely, if $d$ is not a sleek metric for $X$, then in view of Theorem \ref{thA}, there exists a pair of distinct points $u$ and $v$ in $X$ for which $v$ is not a limit point of the set $X\setminus B_d[u,d(u,v)]$. Also, since $v\not\in (X\setminus B_d[u,d(u,v)]$, we must have $v\not\in \overline{X\setminus B_d[u,d(u,v)]}$. Consequently, there exists an $\epsilon>0$, such that $B_d(v,\epsilon)\cap (X\setminus B_d[u,d(u,v)])=\emptyset$. Thus, $B_d(v,\epsilon)\subseteq B_d[u,d(u,v)]$, and so, $v\in B_d^\circ[u,d(u,v)]$. Finally, the map $d(u,\cdot)$ attains its maximum value $d(u,v)$ on the open set $B_d^\circ[u,d(u,v)]$ containing $u$, and so, the converse holds.
\end{proof}
\begin{corollary}
\begin{enumerate}
    \item[(a)] \label{th0}  A metrizable space with at least two points and having an isolated point is not  sleek.
    \item[(b)] \label{th1}    A compact metrizable space having at least two points is not  sleek.
\end{enumerate}
\end{corollary}
\begin{proof}
(a)    Let the metric $d$ induces the topology of the metrizable space $X$. If $x\in X$ is an isolated point, then the set $U=\{x\}$ is open, and so, the map $d(x,\cdot):U\rightarrow\mathbb{R}$ has the maximum value $d(x,x)=0$.

(b) Let the metric $d$ induces the topology of the compact metrizable space $X$. Since $X$ is compact and have at least two points, there exists a pair of distinct points  $x$ and $y$ in $X$, such that $d(x,y)=\sup\{d(u,v)~|~u,v\in X\}$. So, the map $d(x,\cdot):X\rightarrow \mathbb{R}$ has the maximum value $d(x,y)$.
\end{proof}
A compact metrizable space having no isolated point may or may not be metrically round.
An open or dense subspace of a  sleek metric space is   sleek  (see \cite{TH}).
The following result shows that for real line with the usual metric, the class of metrically sleek subspaces is contained in the class of metrically round subspaces.
\begin{theorem}\label{c2}
Let $X$ be a subspace of real line. If $\rho_1$  is a sleek metric for $X$, then $\rho_1$ is a round metric for $X$.
\end{theorem}
\begin{proof} For any $p\in X$, and $r>0$, an open ball centered at $p$ and radius $r$ in $X$ is the set $(p-r,p+r)\cap X$, and the corresponding closed ball is $[p-r,p+r]\cap X$.

Assume that $\rho_1$ is not round for $X$. Then there exists a pair of distinct points $x$ and $y$ in $X$, such that $y\not\in \overline{(x-s,x+s)\cap X}\cap X$, where $s=\rho_1(x,y)>0$ and the bar denotes the closure in real line. Consequently, there exists an $\epsilon>0$, such that  $(y-\epsilon,y+\epsilon)\cap (x-s,x+s)\cap X=\emptyset$. Without loss of generality assume that $x<y$. Then $(y-\epsilon,y)\subset (x-s,x+s)$, and so, $(y-\epsilon,y)\cap X=\emptyset$. Thus $(y-\epsilon,y+\epsilon)\cap X=[y,y+\epsilon)\cap X$. By the hypothesis, $X$ is metrically sleek. So, by Corollary \ref{th0}, $y$ is not an isolated point of $X$. Consequently, we can choose a point $z\in (y,y+\epsilon/2)\cap X$. Then $y,z\in U=[y,2z-y)\cap X$, where $U$ is an open subset of $X$ such that the function $\rho_1(z,\cdot):U\rightarrow\mathbb{R}$ has the maximum value $z-y$. This contradicts Theorem \ref{th4}.
\end{proof}
On the other hand, there always exists a subspace $X$ of the Euclidean space $(\mathbb{R}^n,\rho_n)$ for $n>1$,  such that the metric $\rho_n$ is sleek but not round for $X$. An explicit example is the subspace  $\mathbb{R}\times \{1,2,\ldots,n\}$ of $\mathbb{R}^n$, $n\geq 2$.
\begin{theorem}\label{th2.6}
    For an index set $J$, let $\{A_\alpha\}_{\alpha\in J}$ be a family of subspaces of a metric space $(X,d)$ such that the metric $d$ is sleek for the subspace $A_\alpha\cup A_\beta$ for all $\alpha, \beta\in J$. Then the metric $d$ is sleek for the subspace $\cup_{\alpha\in J}A_\alpha$.
\end{theorem}
\begin{proof}
Let $A=\cup_{\alpha\in J}A_\alpha$. Suppose the contrary that $d$ is not sleek for $A$. Then by Theorem \ref{th4}, there exists $x\in A$ and an open set $U$ containing $x$ such that the map $d(x,\cdot):U\rightarrow \mathbb{R}$ has maximum $d(x,y)$ for $y\in U$. As $U$ is open in $A$, we have $U=O\cap A$ for some open subset $O$ of $X$. Since $x,y\in O\cap A=\cup_{\alpha\in J}(O\cap A_\alpha)$, there exist $\alpha,\beta\in J$ such that $x\in A_\alpha$ and $y\in A_\beta$.  Then $x,y\in O\cap (A_\alpha\cup A_\beta)$, such that $d(x,y)=\sup\{d(x,z)~|~z\in A_\alpha\cup A_\beta\}$, which in view of Theorem \ref{th4} shows that the metric $d$ is not sleek for the subspace $A_\alpha\cup A_\beta$. This contradicts the hypothesis.
\end{proof}
For sleek metric spaces, Theorems \ref{th2}, \ref{th6}, and \ref{th3} are analogues of Theorems \ref{Na2}, \ref{Na3}, and \ref{Na4}, respectively.
\begin{theorem}\label{th2}
Let $(X,d_1)$ and $(Y,d_2)$ be metric spaces. Let $f:X\rightarrow Y$ be a surjection such that for $x,y,z,\in X$, if $d_1(x,z)>d_1(x,y)$, then $d_2(f(x),f(z))>d_2(f(x),f(y))$. If $d_1$ is a sleek metric for $X$, then $d_2$ is a sleek metric for $Y$.
 \end{theorem}
\begin{proof}  The map $f$ is a homeomorphism of $X$ onto $Y$ (see \cite{Na}). So any two distinct points of $Y$ are of the form $f(x)$ and $f(x')$ for  distinct points $x$ and $x'$ in $X$. Since $d_1$ is a sleek metric for $X$, in view of Theorem \ref{thA}, the point $x'$ is a limit point of the set $X\setminus B_{d_1}[x,d_1(x,x')]$. So, there exists a sequence $\{z_n\}$ of points of $X\setminus B_{d_1}[x,d_1(x,x')]$ converging to $x'$. By  the continuity of $f$, the sequence $\{f(z_n)\}$ in $Y$ converges to $f(x')$. Since for each natural number $n$, $z_n\not\in B_{d_1}[x,d_1(x,x')]$, we must have $d_1(x,z_n)>d_1(x,x')$.
So, by the hypothesis we have $d_2(f(x),f(z_n))>d_2(f(x),f(x'))$, which shows that $f(z_n)\not\in B_{d_2}[f(x),d_2(f(x),f(x'))]$. Consequently, $f(x')$ becomes a limit point of the set $Y\setminus B_{d_2}[f(x),d_2(f(x),f(x'))]$. By Theorem \ref{thA}, $d_2$ is a sleek metric for $Y$.
 \end{proof}
 The first two results in the following corollary have been obtained in \cite{Na} for   round metric spaces.
\begin{corollary}
\begin{enumerate}\label{c1}
\item[(a)] Let $d_1$ be a sleek metric for $X$. If $d_2$ is another metric on $X$ equivalent to $d_1$ such that $d_2(x,z)<d_2(x,y)$ whenever $d_1(x,z)<d_1(x,y)$, then $d_2$ is also a sleek metric for $X$.
\item[(b)]
    Let $d$ be a  sleek metric for $X$. Then there exists an equivalent bounded metric $d'$ on $X$ which is sleek for $X$.
\item[(c)]
    Let $f:(X,d_1)\rightarrow (Y,d_2)$ be a global isometry, that is, $f$ is a surjective map such that $d_2(f(x),f(y))=d_1(x,y)$. Then $d_1$ is a sleek metric for $X$ if and only if $d_2$ is a sleek metric for $Y$.
    \end{enumerate}
\end{corollary}
\begin{proof}
To prove (a), we take $X=Y$ and $f$, the identity map of $X$.   To prove (b), let $d'(x,y)=d(x,y)/(1+d(x,y))$ for all $x,y\in X$. The metric $d'$ is bounded and is equivalent to $d$. If $d(x,z)<d(x,y)$, then we have
    \begin{eqnarray*}
    d'(x,y)-d'(x,z)&=&\frac{1}{(1+d(x,y))(1+d(x,z))}\{d(x,y)-d(x,z)\}>0.
    \end{eqnarray*}
By  (a), $d'$ is a sleek metric for $X$. Finally, (c) follows from the observation that $d_1(x,z)>d_1(x,y)$ for $x,y,z\in X$ if and only if $d_2(f(x),f(z))=d_1(x,z)>d_1(x,y)=d_2(f(x),f(y))$.
\end{proof}
\begin{theorem}\label{th6}
    Let $(X,d)$ be a metric space. Then there exists an equivalent bounded metric $d'$ on $X$ which is not sleek for $X$.
\end{theorem}
\begin{proof} As in \cite{Na}, for  $a$, $b\in X$, $a\neq b$ and $0<r<d(a,b)$, we let
$d'(x,y)=\min\{d(x,y),r\}$ for all $x,y\in X$. Then $d'$ is a bounded and is equivalent to $d$.
We observe that $B_{d'}^\circ[a,r]=X\neq B_{d'}(a,r)$, since $b\not\in B_{d'}(a,r)$.
So, $d'$ is not a sleek metric for $X$.
\end{proof}
\begin{theorem}\label{th3}
Let $\{(X_k,d_k)\}_{k=1}^\infty$ be a countable family of  metric spaces, where $\text{diam}(X_k)<\infty$ for all but finitely many $k$.
Let $X=\prod_{k}X_k$. If the metric $d_k$ is sleek for $X_k$ for all $k$, then the metric  $D$ as defined in Theorem \ref{Na4} is sleek for $X$. The converse is not true.
\end{theorem}
\begin{proof} Assume that the metric $d_k$ is sleek for $X_k$ for all $k$. Let $x=(x_k)$ and $y=(y_k)$ be any two distinct points of $X$. Then there exists an index $i$ for which $x_i\neq y_i$. Since the metric $d_i$ is sleek for $X_i$, in view of Theorem \ref{thA}, we have a sequence $\{z_i^{n}\}_{n=1}^{\infty}$ of points of $X_i$ such that
\begin{eqnarray*}
\lim_{n\rightarrow \infty}d_i(z_i^{n},y_i)=0;~d_i(x_i,z_i^{n})>d_i(x_i,y_i),~\text{for all}~n\geq 1.
\end{eqnarray*}
Taking $\xi_k^n=y_k$ if $k\neq i$, and $\xi_k^n=z_i^n$ if $k=i$, we observe that
$\{\xi^{n}\}$ is the sequence of points of $X$ converging to $y$, since $D(\xi^{n},y)=d_i(z_i^{n},y_i)/(\lambda_i2^i)\rightarrow 0$ as $n\rightarrow \infty$. Also,
\begin{eqnarray*}
D(x,\xi^{n}) &=& {d_i(x_i,z_i^{n})}/{(\lambda_i 2^i)}+\sum_{k=1, k\neq i}^\infty {d_k(x_k,y_k)}/{(\lambda_k 2^k)}\\
&=&\{(d_i(x_i,z_i^{n})-d_i(x_i,y_i))\}/(\lambda_i 2^i)+ D(x,y),
\end{eqnarray*}
which shows that $D(x,\xi^{n})>D(x,y)$; and so,  $\xi^n\in X\setminus B_D[x,D(x,y)]$ for all $n$. Thus,  $y$ is a limit point of the set $X\setminus B_D[x,D(x,y)]$. By Theorem \ref{thA},  $D$ is a sleek metric for $X$.

To show that the converse need not be true, we take the product space $X=\prod_{n=1}^\infty (X_n,d_n)$, where
\begin{eqnarray*}
X_1=\{0,1\};~d_1=\rho_1;~X_n=\mathbb{R};~d_n(x_n,y_n)=\frac{|x_n-y_n|}{1+|x_n-y_n|},~x_n,y_n\in X_n, n\geq 2.
\end{eqnarray*}
We show that the metric $D$ as defined in Theorem \ref{Na4} is sleek for the product space $X$. To proceed, we observe that
$\text{diam}(X_n)=1$ for all $n$ so that in view of \eqref{ee1} we have for all $x=(x_n)$, $y=(y_n)$ in $X$   that
\begin{eqnarray}\label{ee2}
D(x,y) &=& \sum_{n=1}^\infty{d_n(x_n,y_n)}{2^{-n}}=\begin{cases}  \sum_{n=2}^\infty\frac{|x_n-y_n|2^{-n}}{1+|x_n-y_n|}, &~\text{if}~x_1=y_1;\\
\frac{1}{2}+\sum_{n=2}^\infty\frac{|x_n-y_n|2^{-n}}{1+|x_n-y_n|}, &~\text{if}~x_1\neq y_1.
\end{cases}
\end{eqnarray}
Consequently, for any point $a=(a_n)\in X$ and $r>0$, we have the following:
\begin{eqnarray*}
B_{D}[a,r]=\begin{cases} A_{r,a_1}, &~\text{if}~0<r<(1/2);\\
A_{r,a_1}\cup \{(1-a_1,a_2,a_3,\ldots)\},&~\text{if}~r=1/2;\\
A_{r,a_1}\cup A_{r-1/2,1-a_1},&~\text{if}~(1/2)<r,
\end{cases}
\end{eqnarray*}
where we have
\begin{eqnarray*}
A_{r,a_1}=\bigl\{(x_n)\in X~|~x_1=a_1,~\sum_{n=2}^\infty\frac{|x_n-y_n|2^{-n}}{1+|x_n-y_n|}\leq r\bigr\}.
\end{eqnarray*}
We claim that $B_D^{\circ}[a,r]=B_D(a,r)$. To prove our claim, we have the following two cases:

\textbf{Case I:} $r\geq 1$. In this case, we observe that $B_D[x,r]=X=B_D(x,r)$ for all $x\in X$; and so, in particular, we have $B_D^\circ[a,r]=B_D(a,r)$.

\textbf{Case II:} $0<r<1$. Let $x=(x_n)\in B_{D}[a,r]$ be such that $D(a,x)=r$. To prove the claim, it will be enough to show that  $x\not\in B_D^{\circ}[a,r]$. In view of this, we arrive at the following subcases:

\textbf{Subcase I:} $x\in A_{r,a_1}$. In this case for any $\epsilon>0$, we choose $y=(y_n)\in X$, where for each natural number $n$,
\begin{eqnarray*}
y_n &=& \begin{cases} x_n,&~\text{if}~x_n=a_n;\\
x_n+\epsilon\frac{x_n-a_n}{|x_n-a_n|},&~\text{if}~x_n\neq a_n.
\end{cases}
\end{eqnarray*}
Now using \eqref{ee2}, we arrive at the following calculations:
\begin{eqnarray*}
D(x,y)=\frac{\epsilon}{1+\epsilon}\sum_{(n\geq2,~x_n\neq a_n)} 2^{-n}\leq \epsilon\sum_{n\geq 2}{2^{-n}}=\epsilon/2<\epsilon,
\end{eqnarray*}
which show that $y\in B_D(x,\epsilon)$. Also, for every nonnegative real number $u$, we have
\begin{eqnarray}\label{ee3}
\frac{u+\epsilon}{1+u+\epsilon}-\frac{u}{1+u}=\frac{\epsilon}{(1+u)(1+u+\epsilon)}>0.
\end{eqnarray}
Using the inequality \eqref{ee3} for $u=|a_n-x_n|$ for each $n$, we get
\begin{eqnarray*}
D(a,y)=
\sum_{(n\geq 2,~x_n\neq a_n)} \frac{(|a_n-x_n|+\epsilon)2^{-n}}{1+|a_n-x_n|+\epsilon}
>\sum_{(n\geq 2,~x_n\neq a_n)} \frac{|a_n-x_n|2^{-n}}{1+|a_n-x_n|}
&=&D(a,x).
\end{eqnarray*}
Consequently, $x\not\in B_D^\circ[a,r]$.

\textbf{Subcase II:} $x\in A_{r-1/2,1-a_1}$ for $r>1/2$. Proceeding as in the \textbf{Subcase II}, we find that $x\not\in B_D^\circ[a,r]$.

\textbf{Subcase III:} $x\in A_{1/2,1-a_1}$. In this case, $x=(1-a_1,a_2,a_3,\ldots)$, and for $\epsilon>0$,
$B_{D}(x,\epsilon)$ contains the point $z=(1-a_1, a_2+\epsilon,a_3,a_4,\ldots)$, where
\begin{eqnarray*}
D(a,z)=\frac{|1-2a_1|}{2}+\frac{\epsilon}{4(1+\epsilon)}.
\end{eqnarray*}
Since $a_1\in \{0,1\}$, we must have $|1-2a_1|=1$; and so, $D(a,z)=(1/2+\epsilon/(4(1+\epsilon)))>(1/2)$. Thus, $x\not\in B_D^\circ[a,1/2]$.

In each of the above subcases, we have $x\not\in B_D^\circ[a,r]$, and so, $B_D^\circ[a,r]=B_D(a,r)$. We conclude that $D$ is a sleek metric for $X$, where we note that the metrizable component $X_1=\{0,1\}$ of $X$ is never sleek.
\end{proof}
\begin{corollary}
A countable product of  sleek metric spaces is  sleek.
\end{corollary}
\begin{proof}
Let $\{(X_k,d_k)\}_{k=1}^\infty$ be a countable collection of metric spaces, where $d_k$ is a sleek metric for $X$ for all $k$. Let $X=\prod_{k=1}^\infty X_i$ be given the product topology. By Corollary \ref{c1}(b), there is an equivalent bounded sleek metric $d'_k$ for $X_k$, and by Theorem \ref{th3}, the product space $\prod_{k=1}^\infty (X_k,d_k')$ has a sleek metric.
\end{proof}
The above result holds for  round metric spaces as well. This was proved by Nathanson \cite{Na}.

The counterexample considered in the converse part of Theorem \ref{th3} suggests the following generalization.
\begin{theorem}\label{th3G}
Let $\{(X_k,d_k)\}_{k=1}^\infty$ be a countable family of  metric spaces, where $\text{diam}(X_k)<\infty$ for all but finitely many $k$. Let $X=\prod_{k}X_k$. Let $D$ be the metric on $X$ as in Theorem \ref{Na4}.
If there  exists at least one positive integer $k$ for which $d_k$ is a sleek metric for $X_k$, then $D$ is a sleek metric for $X$.
\end{theorem}
\begin{proof} Without loss of generality we assume that the metric $d_1$ is sleek for $X_1$. Let $x=(x_n)$ and $y=(y_n)$ be two distinct points of $X$. We show that $y$ is a limit point of the set $X\setminus B_D[x,D(x,y)]$, which in view of Theorem \ref{thA} will prove that the metric $D$ is sleek for $X$. So it is enough to prove that there exists a sequence of points of the set $X\setminus B_D[x,D(x,y)]$ converging to $y$ in the metric $D$. We construct such a sequence in each of the following two cases:

First assume that $x_1=y_1$. We define a sequence $\{\xi_n\}$ of points of $X\setminus B_D[x,D(x,y)]$ as follows. Since $X_1$ is sleek, by Corollary \ref{th0} no point of $X_1$ is an isolated point. So we can choose a nonconstant sequence $\{x_1^{(n)}\}$ of points of $X_1$ converging to $x_1$, that is, $\lim_{n\rightarrow \infty}d_1(x_1^{(n)},x_1)=0$. Without loss of generality, we can assume that $x_1^{(n)}\neq x_1$ for all $n$ so that $d_1(x_1^{(n)},x_1)$ is always positive. Let $\xi_n=(x_1^{(n)},y_2,y_3,\ldots)\in X$ for all $n$. Then
$D(\xi_n,y)=\frac{1}{2\lambda_1}d_1(x_1^{(n)},x_1)\rightarrow 0$ as $n\rightarrow\infty$, which shows that the sequence $\{\xi_n\}$ converges to $y$.  Also, for every positive integer $n$, we have
\begin{eqnarray*}
D(\xi_n,x)-D(x,y)=\frac{1}{2\lambda_1}d_1(x_1^{(n)},y_1)>0,
\end{eqnarray*}
which shows that $\xi_n\in X\setminus B_D[x,D(x,y)]$ for all $n$.

Now assume that  $x_1\neq y_1$. Since the metric $d_1$ is sleek for $X_1$, by Corollary \ref{th0} the point $y_1$ is a limit point of the set $X_1\setminus B_{d_1}[x_1,d_1(x_1,y_1)]$. Consequently, there exists a sequence $y_1^{(n)}$ of points of $X_1$ such that $d_1(y_1^{(n)},y_1)\rightarrow 0 $ as $n\rightarrow \infty$, and $d_1(y_1^{(n)},x_1)>d_1(y_1,x_1)$ for all $n$. Here we take $\eta_n=(y_1^{(n)},y_2,y_3,\ldots)\in X$ for all $n$. Then $D(\eta_n,y)=\frac{1}{2\lambda_1}d_1(y_1^{(n)},y_1)\rightarrow 0$ as $n\rightarrow \infty$; and for each positive integer $n$
\begin{eqnarray*}
D(\eta_n,x)-D(x,y)=\frac{1}{2\lambda_1}\{d_1(y_{1}^{(n)},x_1)-d_1(y_{1},x_1)\}>0.
\end{eqnarray*}
Thus, $\{\eta_n\}$ is a sequence of points of $X\setminus B_D[x,D(x,y)]$ converging to $y$.
\end{proof}
\begin{corollary}
Let $\{(X_k,d_k)\}_{k=1}^\infty$ be a countable family of  metric spaces, such that $d_k$ is a sleek metric for $X_k$ for at least one value of $k$.
Then the product space $\prod_{k}X_k$ is  sleek.
\end{corollary}
\begin{examples}
1.  \label{ex1}
  In the dictionary order topology, the topological space $\mathbb{R}^2$ can be identified with the (metrizable) product space $\mathbb{R}_{\text{dis}}\times \mathbb{R}$, where $\mathbb{R}_{\text{dis}}$ and $\mathbb{R}$ respectively denote the set of all real numbers with the discrete metric and the usual metric. Let $\rho$ be the product metric for $\mathbb{R}_{\text{dis}}\times \mathbb{R}$.   Observe that for any two distinct real numbers $x$ and $y$, the metric $\rho$ is sleek for each of the subspaces $\{x\}\times \mathbb{R}$, $\{y\}\times \mathbb{R}$,  and their union $\{x,y\}\times \mathbb{R}$. By Theorem \ref{th2.6}, the metric $\rho$ is sleek for the subspace $\cup_{x\in \mathbb{R}}\{x\}\times \mathbb{R}=\mathbb{R}_{\text{dis}}\times \mathbb{R}$.

2. Let $(X,d)$ be a metric space.  Let $\phi:[0,\infty)\rightarrow [0,\infty)$ be a strictly increasing function such that the composite map $\phi\circ d$ is a metric on $X$. If the metric $d$ is sleek for $X$, then in view of Corollary \ref{c1}(a), the metric $\phi\circ d$ is sleek for $X$. An explicit example of such a function $\phi$ is
\begin{eqnarray*}
\phi(t)=\log(1+t) ~\text{for all}~t\in [0,\infty).
\end{eqnarray*}

3. The  metric $\rho_1$  is not sleek for any of the subspaces $[0,1]$ and $(0,1]$ of real line but the metric $\rho_2$ is sleek for each of the product spaces $\mathbb{R}\times [0,1]$ and $\mathbb{R}\times (0,1]$ of $\mathbb{R}^2$.
\end{examples}
\section{Applications to strictly convex linear metric spaces.}
Recall that a linear metric space is a topological vector space with a compatible translation invariant metric.
We now investigate strictly convex linear metric spaces for round and sleek translation invariant metrics.
A linear metric space $(X,d)$ is said to be \emph{strictly convex} \cite{ahuja} if for any $r>0$ and any two distinct points $x,y\in X$ such that $d(x,0)\leq r$ and $d(y,0)\leq r$, we have $d((x+y)/2,0)<r$. The closed ball $B_d[0,r]$ in the linear metric space $(X,d)$ is said to be strictly convex \cite{vasilev} if for any pair of distinct points $x$ and $y$ in $B_d[0,r]$ and $\lambda\in (0,1)$, the point $(1-\lambda)x+\lambda y$ belongs to $B_d^\circ[0,r]$.
\begin{remark}
If a linear metric space $(X,d)$ is strictly convex, then with the equivalent bounded linear metric $d'$ on $X$ defined by $d'(x,y)=d(x,y)/(1+d(x,y))$ for all $x,y\in X$, the linear metric space $(X,d')$ is strictly convex.
\end{remark}
The following two characterizations of strict convexity corresponding to round and sleek properties were proved in \cite{vasilev} and \cite{JSTD2020} respectively.
\begin{thmx}\label{thB} In a linear metric space $(X,d)$, the following statements are equivalent:
\begin{enumerate}
\item  The space $(X,d)$ is strictly convex.
\item All closed balls in $X$ are strictly convex, and the metric $d$ is round for $X$.
\item All closed balls in $X$ are strictly convex, and the metric $d$ is sleek for $X$.
\end{enumerate}
\end{thmx}
So, in a strictly convex linear metric space having strict ball convexity, the notions of being round and being sleek are equivalent.
\begin{remarks}
\begin{enumerate}
\item Let $(X,d)$ be a linear metric space. Then there exists an equivalent bounded translation-invariant metric $d'$ on $X$ for which $(X,d')$ is not strictly convex.
\begin{proof} Let $0\neq b\in X$ and $0<r<d(0,b)$. Let
$d'(x,y)=\min\{r,d(x,y)\}$ for all $x,y\in X$.
Then $d'$ is an equivalent bounded metric for $X$. Since $d$ is translation-invariant and the function $\min$ is continuous, the metric $d'$ is translation invariant. As in the proof of Theorem \ref{th6}, we see that the translation-invariant metric $d'$ is not sleek. So, by Theorem \ref{thB}, the linear metric space $(X,d')$ is not strictly convex.
\end{proof}
\item Let $(X_i,d_i)_{k=1}^n$ be a collection of $n$ metric spaces. The product topology of $\prod_{k=1}^nX_k$ can be induced by the metric $d$, where
\begin{eqnarray*}
d(x,y) =\sqrt{\sum_{i=1}^n(d_i(x_i,y_i))^2},
\end{eqnarray*}
for all $x=(x_1,\ldots,x_n)$ and $y=(y_1,\ldots,y_n)$ in $\prod_{k=1}^nX_k$.
Now if $(X_k,d_k)$ is a linear metric space for all $k$, then the product space $(\prod_{k=1}^nX_k,d)$ is also a linear metric space.
If $(X_k,d_k)$ is a strictly convex linear metric space for all $k$, then $(\prod_{k=1}^nX_k,d)$ need not be strictly convex (see \cite[Example 2.2]{SN81}).
\item Let $\{(X_k,d_k)\}_{k=1}^\infty$ be a countable family of linear metric spaces, such that $\text{diam}(X_k)<\infty$ for all but finitely many $k$. Let $X=\prod_{k}X_k$, and the metric $D$ as in Theorem \ref{Na4}. Let $(X,D)$ and $(X_k,d_k)$ for each $k$  have strict ball convexity.  Then $(X,D)$ is strictly convex if and only if $(X_k,d_k)$ is strictly convex for all $k$.
\begin{proof}
In view of Theorem \ref{thB}, the metric space $(X,D)$ is strictly convex $\Leftrightarrow$ the metric $D$ is round for $X$ $\Leftrightarrow$  the metric $d_k$ is round for $X_k$ for all $k$ as follows from Theorem \ref{Na4} $\Leftrightarrow$ the metric space $(X_k,d_k)$ is strictly convex for all $k$ as follows from Theorem \ref{thB}. This completes the proof.
\end{proof}
\item Let $\{(X_k,d_k)\}_{k=1}^\infty$ be a countable collection of strictly convex linear metric spaces. Then there exists a bounded translation-invariant metric for $X$ under which the product space $\prod_{k}X_k$ becomes a strictly convex linear metric space.
\end{enumerate}
\end{remarks}
\subsection*{Acknowledgments.}
The present research is supported by Science and Engineering Research Board(SERB), a
statutory body of Department of Science and Technology (DST), Govt. of India through the project grant
no. MTR/2017/000575 awarded to the first author under the MATRICS Scheme.

\end{document}